\documentclass[number,citesort,MSNbibl,dvips]{arxbj}
\usepackage{mathbh,upgreek}
\usepackage{graphicx}

\aid{0}
\volume{17}
\issue{3}
\pubyear{2011}
\firstpage{1015}
\lastpage{1043}
\doi{10.3150/10-BEJ301}

\makeatletter
\newcommand{\stsets}[1]{\mathbb{#1}}
\newcommand{\R}{\stsets{R}}
\newcommand{\M}{\stsets{M}}
\newcommand{\Z}{\stsets{Z}}
\renewcommand{\SS}{\stsets{S}}
\newcommand{\eqref}[1]{(\ref{#1})}

\newtheorem{Th}[definition]{Theorem}
\newtheorem{Cor}[definition]{Corollary}
\newproclaim{definition}{Definition}
\newremark{remark}[definition]{Remark}
\newremark{example}[definition]{Example}
\renewcommand{\P}{\mathbf{P}}
\newcommand{\Q}{\mathbf{Q}}
\newcommand{\E}{\mathbf{E}}

\newcommand{\supp}{\operatorname{supp}}

\newcommand{\ind}{\mathbh{1}}
\newcommand{\one}{{\mathbf{1}}}

\newcommand{\cond}{\mid}

\newcommand{\thru}{,\ldots,}
\newcommand{\bydef}{\stackrel{\mathrm{def}}{=}}

\renewcommand{\epsilon}{\varepsilon}
\renewcommand{\phi}{\varphi}
\newcommand{\ti}{\to\infty}
\newcommand{\seg}{see, e.g., }

\newcommand{\cf}{cf. }
\newcommand{\iid}{i.i.d. }
\newcommand{\as}{a.s. }
\newcommand{\pgf}{p.g.f. }
\newcommand{\pgfl}{p.g.fl. }

\newcommand{\sB}{\mathcal{B}}
\newcommand{\sN}{\mathcal{N}}
\newcommand{\sP}{\mathcal{P}}
\newcommand{\sas}{\mathrm{St\alpha S}}
\newcommand{\das}{\mathrm{D\alpha S}}
\newcommand{\deq}{\stackrel{\mathcal{D}}{=}}

\newcommand{\eps}{\varepsilon}
\newcommand{\neutral}{\mathbf{e}}

\newcommand{\cz}{\mathrm{BM}(X)}

\newcommand{\Bin}{\mathsf{Bin}}
\newcommand{\Sib}{\mathsf{Sib}}
\makeatother

\begin{document}
\begin{frontmatter}

\title{Stability for random measures, point processes and discrete semigroups}
\runtitle{Stability for random measures, point processes and discrete
semigroups}

\begin{aug}
\author[a]{\fnms{Youri} \snm{Davydov}\thanksref{a}\ead[label=e1]{youri.davydov@univ-lille1.fr}},
\author[b]{\fnms{Ilya} \snm{Molchanov}\thanksref{b}\ead[label=e2]{ilya@stat.unibe.ch}}
\and
\author[c]{\fnms{Sergei} \snm{Zuyev}\corref{}\thanksref{c}\ead[label=e3]{sergei.zuyev@chalmers.se}}

\runauthor{Y. Davydov, I. Molchanov and S. Zuyev}
\address[a]{Laboratoire de Paule Painlev\'e, UFR de Math\'ematiques,
Universit\'e Lille~1, Villeneuve d'Ascq Cedex, France. \printead{e1}}

\address[b]{Department of Mathematical
Statistics and Actuarial Science, University of Berne, Sidlerstrasse 5,
\mbox{CH-3012} Berne, Switzerland. \printead{e2}}

\address[c]{Department of Mathematical Sciences, Chalmers University
of Technology, 412 96 Gothenburg, Sweden. \printead{e3}}
\end{aug}

\received{\smonth{10} \syear{2009}}
\revised{\smonth{3} \syear{2010}}

%
\begin{abstract}
Discrete stability extends the classical notion of stability to
random elements in discrete spaces by defining a scaling operation
in a randomised way: an integer is transformed into the
corresponding binomial distribution.
Similarly defining the scaling operation as thinning of counting
measures we characterise the corresponding discrete stability
property of point processes. It is shown that these processes are
exactly Cox (doubly stochastic Poisson) processes with
strictly stable random intensity measures. We give
spectral and LePage representations for general strictly stable
random measures without assuming their independent scattering. As a
consequence, spectral representations are obtained for the probability
generating functional and void probabilities of
discrete stable processes. An alternative cluster representation for
such processes is also derived using the so-called Sibuya
point processes, which constitute a new family of purely random point
processes. The obtained results are then applied to explore stable
random elements in discrete semigroups, where the scaling is defined
by means of thinning of a~point process on the basis of the
semigroup. Particular examples include discrete stable vectors that
generalise discrete stable random variables and the family of
natural numbers with the multiplication operation, where the primes
form the basis.
\end{abstract}

%

\end{frontmatter}

\section{Introduction}
\label{sec:introduction}

Stability for random variables was introduced by Paul
L\'evy and thereafter became one of the key concepts in probability. A
random vector $\xi$ (or its probability law) is called \textit{strictly
$\alpha$-stable} (notation: $\sas$) if for any positive numbers $a$
and $b$ the following identity is satisfied:
\[
a^{1/\alpha}\xi'+b^{1/\alpha}\xi''\deq(a+b)^{1/\alpha}\xi ,
\]
where $\xi', \xi''$ are independent vectors distributed as $\xi$ and
$\deq$ denotes equality in distribution. Non-trivial distributions
satisfying this exist only for $\alpha\in(0,2]$. Using $t=a/(a+b)$ in
the above
definition, the stability property can be equivalently expressed as
%
\begin{equation}
\label{eq:defsas}
t^{1/\alpha}\xi'+(1-t)^{1/\alpha}\xi''\deq\xi
\end{equation}
for any $t\in[0,1]$.

The recent work Davydov, Molchanov and Zuyev \cite{davmolzuy08} explored the notion of strict
stability for a much more general situation of random elements taking
values in a commutative semigroup $X$. The stability property relies on
two basic operations defined on $X$: the semigroup operation, addition,
and the rescaling that plays a role of multiplication by
numbers. Many properties of the classical stability still hold for
this general situation, but there are many notable differences related
to specific algebraic properties of $X$; for instance, the validity of
distributivity laws, relation between the zero and the neutral
elements, compactness, etc. Most important, the stable random
elements arise as a sum of points of a Poisson process whose intensity
measure has a specific scalable form. In the classical case of
$d$-dimensional vectors, this result turns into the \textit{LePage
representation} of strictly $\alpha$-stable laws: any such law
corresponds to the distribution of the sum of points of a Poisson
process $\Pi_\alpha$ in $\R^d$ with the density (intensity) function
having a~pro\-duct form $\theta_\alpha(\mathrm{d}\rho)\sigma(\mathrm{d}s)$ in the polar
coordinates $(\rho,s)$. Here $\sigma$ is any finite measure on the
unit sphere $\SS^{d-1}$, called the \textit{spectral measure}, and
$\theta_\alpha((r,\infty))=r^{-\alpha}$. The underlying reason is that
the process $\Pi_\alpha$ is itself stable with respect to scaling and
superposition, that is,
%
\begin{equation}
\label{eq:piscale}
t^{1/\alpha}\Pi_\alpha'+(1-t)^{1/\alpha}\Pi_\alpha''\deq
\Pi_\alpha ;
\end{equation}
this property being granted by the measure $\theta_\alpha$ that scales
in a `right' way and the superposition property of Poisson processes.
This fundamental stability property leads to a characterisation of
strictly stable laws in very general Abelian semigroups forming a cone
with respect to rescaling with a continuous argument $t\geq0$;
see~\cite{davmolzuy08}.

The case of discrete spaces, however, cannot be treated the same way
since the scaling by a continuous argument cannot be defined in such
spaces. While infinite divisibility of random elements in general
semigroups, at least in the commutative case, is well understood (see
\cite{rus88a,rus88b}), a systematic exploration of stable laws on
possibly discrete semigroups is not available. This prompts us to define
rescaling as a direct transformation of probability distributions
rather than being inherited from rescaling of the underlying phase
space of the random elements.

Prior to the current work, the family $\Z_+$ of non-negative integers
with addition was the only discrete semigroup for which stability was
defined. The discrete stability concept for non-negative integer-valued
random variables was introduced by
\cite{stehar79}, who defined the result $t\circ n$ of
rescaling of $n\in\Z_+$ by $t\in[0,1]$ to be the binomial probability
measure $\Bin(n,t)$ with the convention that $0\circ n=0$ for any $n$.
Since $\Bin(n,t)$ corresponds to the sum of $n$ independent Bernoulli
random variables $\Bin(1,t)$ with parameter $t$, one can view $t\circ
n$ as the total count of positive integers between $1$ and $n$ where
each number is counted with probability $t$ independently of others.
Thus a $\Z_+$-valued random variable~$\xi$
is mapped to a random variable $t\circ\xi$ having distribution of
$\sum_{n=0}^\xi\beta_n$, where $\{\beta_n\}$ is a~sequence of
independent $\Bin(1,t)$ random variables. Thus $t\circ\xi$ can also
be viewed as a~doubly stochastic random variable whose distribution
depends on realisations of $\xi$. This multiplication operation, though
in a different context, goes back to~\cite{Renyi57}.

In terms of the probability generating function (p.g.f.), if
$g_\xi(s)=\E s^\xi$ denotes the \pgf of $\xi$, then the \pgf of
$t\circ\xi$ is given by the composition of $g_\xi$ and the \pgf of the
Bernoulli $\Bin(1,t)$ law:
%
\begin{equation}
\label{eq:tphi}
g_{t\circ\xi}(s)=g_\xi\bigl(1-t(1-s)\bigr) .
\end{equation}

Similarly to \eqref{eq:defsas}, a random variable $\xi$ is called
\textit{discrete $\alpha$-stable} if for all $0\leq t\leq1$ the
following identity is satisfied
%
\begin{equation}
\label{eq:dstabdef}
t^{1/\alpha}\circ\xi'+(1-t)^{1/\alpha}\circ\xi''\deq
\xi ,
\end{equation}
where $\xi',\xi''$ are independent copies of $\xi$. The full
characterisation of discrete
$\alpha$-stable random variables (denoted by $\das$ in the sequel) is
provided in \cite{stehar79} namely the laws
satisfying (\ref{eq:dstabdef}) exist only for $\alpha\in(0,1]$ and
each such law has a \pgf of the form
%
\begin{equation}
\label{eq:lawdas}
\E s^\xi=\exp\{-c(1-s)^\alpha\}
\end{equation}
for some $c>0$.

It has also been shown that the above multiplication, defined using the
Binomial distribution, can be embedded in the family of multiplication
operations that correspond to branching processes or, alternatively,
to semigroups of probability generating functions; see
\cite{harsteutver82}. A variant of this operation for
integer-valued vectors with multiplication defined by operator-scaling
of the probability generating function was studied by
Van Harn and Steutel \cite{harsteut86} in view of queueing applications. This setting
has been extended to some functional equations for discrete random
variables in \cite{harsteut93,pak95}.

In this paper we show that the discrete stability of non-negative
integer random variables is a particular case of stability with
respect to the thinning operation of point processes defined on a
rather general phase space. If the phase space consists of one
element, a point process may have multiple points -- the total number
of which
is a non-negative integer random variable -- and the thinning of the
points is equivalent to the above-defined scaling operation $\circ$ on
discrete random variables. All known properties of discrete stable
laws have their more general counterparts in point process settings.
In particular, we show in Section \ref{sec:stab-with-resp} that
$\alpha$-stable point processes with respect to thinning are, in fact,
Cox processes (\seg Chapter 6.2 in \cite{DalVJon02}) with the
(random) parametric measure being a positive $\alpha$-stable measure
with $\alpha\in(0,1]$ on the phase space. Because of this, we
first address general strictly stable random measures in
Section \ref{sec:strictly-stable-rand}. Note in this relation that, so
far, only \textit{independently scattered} stable measures have received
much attention in the literature; see \cite{SamTaq94} and, more
recently, \cite{davegor00,hell09} on the subject. It
should be noted that Cox processes driven by various random measures
are often used in spatial statistics (see
\cite{helprokved08,mol03,molsyvwaa98}). A particularly novel feature
of discrete stable processes is that the point counts have discrete
$\alpha$-stable distributions, which have infinite expectations unless
for the degenerate case of $\alpha=1$. Thus, these point processes
open new possibilities for modelling point patterns with
non-integrable point counts and non-existing moment measures of all
orders.

It is known that discrete stable random variables can be represented
as the sum of a Poisson number of Sibuya distributed integer random
variables; see \cite{Dev92}. In Section \ref{sec:poiss-mixt-repr} we
show that an analogous cluster representation of finite discrete
stable point processes also holds. The clusters are Sibuya point
processes, which seem to be a new class of point processes not
considered so far.

Some of our results for general point processes, especially defined on
non-compact spaces, become trivial or do not have counterparts for discrete
random variables. This concerns the results of
Section \ref{sec:infinite-das-point}, which draws an analogy with
infinitely divisible point processes.

Another important model arising from the point process setting
is the case of random elements in discrete semigroups that possess an at
most countable basis. We show in Section \ref{sec:settings} that
the uniqueness of representation of each element as a linear finite
combination with natural coefficients of the basis elements is a
necessary condition to be able to define discrete stability. We give
characterisation of $\das$ elements in these semigroups and establish
a discrete analogue of their LePage representation.

The presentation of the theoretical material is complemented by
examples from random measures, point processes and discrete semigroups
all over the text. In particular, in Section \ref{sec:settings} we
introduce a new concept of multiplicatively stable natural numbers
that is interesting on its own right and
characterise their distributions.

\section{Strictly stable random measures}
\label{sec:strictly-stable-rand}

Let $X$ be a locally compact second countable space with the Borel
$\sigma$-algebra $\sB$. The family of Radon measures on $\sB$ is
denoted by $\M$. These are measures that have finite values on the
family $\sB_0$ of relatively compact Borel sets, that is, the sets with
compact topological closure. Note that all measures considered in
this paper are assumed to be non-negative. The zero measure is denoted
by $0$.

A \textit{random measure} is a random element in the measurable space
$[\M,\mathcal{M}]$, where the $\sigma$-algebra $\mathcal{M}$ is
generated by the following system of sets:
\[
\{\mu\in\M\dvtx \mu(B_i)\leq t_i,  i=1\thru n\},\qquad
B_i\in\sB, t_i\geq0 .
\]
The distribution of a random measure $\zeta$ can be characterised by
its finite-dimensional distributions, that is, the distributions of
$(\zeta(B_1)\thru\zeta(B_n))$ for each finite collection of disjoint
sets $B_1\thru B_n\in\sB_0$.

It is well known (see, e.g., \cite{DalVJon08}, Chapter 9.4) that the distribution of
$\zeta$ is uniquely determined by its \textit{Laplace functional}
%
\begin{equation}
\label{eq:lap}
L_\zeta[h]=\E\exp \biggl\{ - \int_X h(x) \zeta(\mathrm{d}x) \biggr\}
\end{equation}
defined on non-negative bounded measurable functions $h:\ X\mapsto
\R_+$ with compact support (denoted by $h\in\cz$). From now on we
adapt a
shorter notation
\[
\langle h,\mu\rangle=\int_X h(x) \mu(\mathrm{d}x) ,\qquad  \mu\in\M .
\]
%

The family $\M$ can be endowed with the operation of addition and
multiplication by non-negative numbers as
%
\begin{eqnarray}
(\mu_1+\mu_2)(\cdot)&=&\mu_1(\cdot)+\mu_2(\cdot) ; \label{eq:sumop}\\[3pt]
(t\mu)(\cdot)&=&t\mu(\cdot) , \qquad  t\geq0 . \label{eq:multop}
\end{eqnarray}
If $\mu$ is a finite non-negative measure, it is possible to normalise
it by
dividing it by its total mass and so arriving at a probability measure.
The normalisation procedure can be extended to all locally finite measures
as follows. Let $B_1,B_2,\ldots$ be a fixed countable base of the
topology on $X$ that consists of relatively compact sets. Append
$B_0=X$ to this base. For each $\mu\in\M\setminus\{0\}$ consider
the sequence of
its values $(\mu(B_0),\mu(B_1),\mu(B_2),\ldots),$ possibly starting
with infinity, but otherwise finite. Let $i(\mu)$ be the smallest
non-negative integer $i$ for which $0<\mu(B_i)<\infty$; in particular,
$i(\mu)=0$ if $\mu$ is a finite measure. The set\looseness=1
%
\begin{equation}
\label{eq:s-def}
\SS=\bigl\{\mu\in\M\dvtx  \mu\bigl(B_{i(\mu)}\bigr)=1\bigr\}
\end{equation}
is measurable, since
\[
\SS=\M_1\cup\bigcup_{n=1}^\infty\{\mu\in\M\dvtx  \mu(B_0)=\infty
, \mu(B_1)
=\cdots=\mu(B_{n-1})=0, \mu(B_n)=1\} ,
\]
where $\M_1$ is the family of all probability measures on $X$. Note
that $\SS\cap\{\mu\dvtx \mu(X)<\infty\}=\M_1$. Furthermore, every
$\mu\in\M\setminus\{0\}$ can be uniquely associated with the pair
$(\hat{\mu},\mu(B_{i(\mu)}))\in\SS\times\R_+$, so that
$\mu=\mu(B_{i(\mu)})\hat{\mu}$. It is straightforward to check that
the mapping $\mu\mapsto(\mu(B_{i(\mu)}),\hat{\mu})$ is measurable.
Hence we have the following \textit{polar decomposition}: $\M=\SS
\times
\R_+$.

\begin{definition}
\label{def:stm}
A random measure $\zeta$ is called \textit{strictly}
$\alpha$-\textit{stable} (notation $\sas$) if
%
\begin{equation}
\label{eq:stabmes}
t^{1/\alpha}\zeta'+(1-t)^{1/\alpha}\zeta''\deq\zeta
\end{equation}
for all $0\leq t \leq1$, where $\zeta',\zeta''$ are
independent copies of the random measure $\zeta$.
\end{definition}

Definition \ref{def:stm} yields that for any $B_1\thru B_n\in\sB_0$,
the vector $(\zeta(B_1)\thru\zeta(B_n))$ is a non-negative $\sas$
$n$-dimensional random vector implying that $\alpha\in(0,1]$. It is
well known that one-sided -- that is, concentrated on $\R_+$ --
strictly stable laws corresponding to
$\alpha=1$ are degenerated so that $\sas$ measures with $\alpha=1$
are deterministic; see, for example,  \cite{SamTaq94}.

\begin{Th}
\label{thr:gen-stas}
A locally finite random measure $\zeta$ is $\sas$ if and only if
$\zeta$ is deterministic in the case $\alpha=1$ and in the case
$\alpha\in(0,1)$ if and only if its Laplace functional is given by
%
\begin{equation}
\label{eq:lf}
L_\zeta[h]=\exp \biggl\{-\int_{\M\setminus\{0\}}
\bigl(1-\mathrm{e}^{-\langle h,\mu\rangle}\bigr) \Lambda(\mathrm{d}\mu) \biggr\} ,\qquad
  h\in\cz ,\vspace*{2pt}
\end{equation}
where $\Lambda$ is a L\'evy measure, that is, a Radon measure on
$\M\setminus\{0\}$ such that
%
\begin{equation}
\label{eq:lmreg}
\int_{\M\setminus\{0\}}\bigl(1-\mathrm{e}^{-\langle h,\mu\rangle}\bigr)
\Lambda(\mathrm{d}\mu)<\infty\vspace*{-2pt}
\end{equation}
for all $h\in\cz$, and $\Lambda$ is homogeneous of order $-\alpha$,
that is,
$\Lambda(tA)=t^{-\alpha}\Lambda(A)$ for all measurable
$A\subset\M\setminus\{0\}$ and $t>0$.
\end{Th}
\begin{pf}
\textit{Sufficiency}. It is obvious that each deterministic measure
is $\sas$ with \mbox{$\alpha=1$}. Consider now $h_1,\dots,h_k\in\cz$ and
the image $\tilde{\Lambda}$ of $\Lambda$ under the map
$\mu\mapsto(\langle h_1,\mu\rangle,\dots,\langle h_k,\mu\rangle)$.
It is easy to see that $\tilde\Lambda$ is a homogeneous measure on
$\R_+^k$. Then $L_\zeta[\sum_{i=1}^k t_ih_i]$ as a function of
$t_1,\dots,t_k$ is the Laplace transform of a totally skewed to the
right (i.e., almost surely positive) strictly stable random vector
$(\langle h_1,\zeta\rangle,\dots,\langle h_k,\zeta\rangle)$ with the
L\'evy measure $\tilde\Lambda$. If we show that (\ref{eq:lf})
defines a Laplace functional of a random measure this would imply
its stability.

We have $L_\zeta[0]=1$ and $L_\zeta[h_n]\to L_\zeta[h]$ if
$h_n\uparrow h$ pointwise. To see this, $H_n(\mu)=\langle
h_n,\mu\rangle\uparrow\langle h,\mu\rangle=H(\mu)$ by the monotone
convergence theorem. In its turn, $\langle H_n,\Lambda\rangle$
converges to $\langle H,\Lambda\rangle$ again by the monotone
convergence. Now, by an analogue of \cite{DalVJon08}, Theorem~9.4.II, for
the Laplace functional, (\ref{eq:lf}) indeed is the Laplace
functional of a random measure.\looseness=1

\textit{Necessity}. Strictly stable random measures can be treated
by means of the general theory of strictly stable laws on semigroups
developed in \cite{davmolzuy08}. Consider the cone $\M$ of locally
finite measures with the addition and scaling operation defined by
(\ref{eq:sumop}) and (\ref{eq:multop}). In the terminology of
\cite{davmolzuy08}, this set becomes a pointed cone with the origin
and neutral element being the zero measure. The second
distributivity law holds and the gauge function
$\|\mu\|=\mu(B_{i(\mu)})$, $\mu\in\M$, is homogeneous with respect
to multiplication of measures by a~number. As in
\cite{davmolzuy08}, Example 8.6, we argue that the characters on $\M$
\[
\vspace*{2pt}
\chi_h(\mu)=\exp \{-\langle h,\mu\rangle \},\qquad    h\in
\cz ,
\vspace*{2pt}
\]
have a strictly separating countable subfamily and are continuous,
so that condition (C) of \cite{davmolzuy08} holds.
Furthermore, $\chi_h(s\mu)\to1$ as $s\downarrow0$ for all
$\mu\in\M$, so that condition (E) of \cite{davmolzuy08}
is also satisfied. By Davydov, Molchanov and Zuyev \cite{davmolzuy08}, Theorem 5.16,
the log-Laplace transform of a $\sas$ measure satisfies
$L_\zeta[sh]=s^\alpha L_\zeta[h]$ for $s>0$, where $\alpha\in(0,1]$
by Davydov, Molchanov and Zuyev \cite{davmolzuy08}, Theorem 5.20(ii). If $\alpha=1$, this identity
implies that the values of $\zeta$ are deterministic. By
Davydov, Molchanov and Zuyev \cite{davmolzuy08}, Theorems 6.5(ii) and 6.7(i), the log-Laplace
functional of a $\sas$ random measure with $\alpha\in(0,1)$ can be
represented as the integral similar to (\ref{eq:lf}) with the
integration taken with respect to the L\'evy measure of $\zeta$ over
the second dual semigroup, that is,\ the family of all characters acting
on functions from $\cz$.

In order to reduce the integration domain to $\M\setminus\{0\}$, it
suffices to show that the convergence $\langle h,\mu_k\rangle\to
g(h)$ for all $h\in\cz$ implies that $g(h)=\langle h,\mu\rangle$ for
some $\mu\in\M$. Taking $h=\ind_B$ here implies convergence and
hence boundedness of the sequence $\{\mu_k(B),  k\geq1\}$ for each
$B\in\sB_0$. By Daley and Vere-Jones \cite{DalVJon02}, Corollary A2.6.V, the sequence
$\{\mu_k,  k\geq1\}$ is relatively compact in the vague topology.
Given that $\langle h,\mu_k\rangle$ converges for each $h\in\cz$, in
particular for all continuous $h\in\cz$, we obtain that all vaguely
convergent subsequences of $\{\mu_k,  k\geq1\}$ share the same
limit that can be denoted by~$\mu$ and used to represent
$g(h)=\langle h,\mu\rangle$ for continuous $h$. However, convergence
$\mu_n\to\mu$ also happens in the strong local topology
corresponding to
test functions $h$ from whole $\cz$. Indeed, if two subsequences
$\mu_{i_1},\mu_{i_2},\dots$ and $\mu_{j_1},\mu_{j_2},\dots$ have
different limits in this topology, say, $\mu'$ and $\mu''$, then
there is $h\in\cz$ such that $\langle h,\mu'\rangle\neq\langle
h,\mu''\rangle$. Then for the subsequence
$\mu_{i_1},\mu_{j_1},\mu_{i_2},\mu_{j_2},\dots$ (possibly after
removing the repeating members) the limit of the integrals of $h$
does not exist, contradicting the assumption. Now (\ref{eq:lf})
follows from \cite{davmolzuy08}, Theorem 7.7.
\end{pf}

The properties of $\sas$ measure $\zeta$ are determined by its L\'evy
measure $\Lambda$. For instance, if $X=\R^d$, then $\zeta$ is
stationary if and only if $\Lambda$ is invariant with respect to the
map $\mu(\cdot)\mapsto\mu(\cdot+a)$ for all $a\in\R^d$ and all
$\mu$
from the support of $\Lambda$. Note also that
(\ref{eq:lmreg}) is equivalent to
%
\begin{equation}
\label{eq:lm-new}
\int_{\M\setminus\{0\}}\bigl(1-\mathrm{e}^{-\mu(B)}\bigr)\Lambda(\mathrm{d}\mu)<\infty
\end{equation}
for all $B\in\sB_0$.

Because of homogeneity, it is also useful to decompose $\Lambda$ into
the radial and directional components using the polar decomposition of
$\M$
described above. Note that the map $\mu\mapsto
(\hat{\mu},\mu(B_{i(\mu)}))$ is measurable. Introduce a measure
$\widehat{\sigma}$ such that
\[
\widehat{\sigma} (A)=\Lambda(\{t\mu\dvtx  \mu\in A,  t\geq1\})
\]
for all measurable $A\subset\SS$. Then
$\Lambda(A\times[a,b])=\widehat{\sigma}(A)(a^{-\alpha}-b^{-\alpha})$,
so that $\Lambda$ is represented as the product of $\widehat{\sigma}$
and the radial component given by the measure $\theta_\alpha$ defined
as $\theta_\alpha([r,\infty))=r^{-\alpha}$, $r>0$. By the reason which will become apparent in the proof of
Theorem \ref{th:finite} below, it
is more convenient to scale $\widehat{\sigma}$ by the value
$\Gamma(1-\alpha)$ of the gamma function. The measure
$\sigma=\Gamma(1-\alpha)\widehat{\sigma}$ will be called the \textit
{spectral
measure} of $\zeta$ in the sequel. Note that $\sigma$ is not
necessarily finite unless~$X$ is compact.

Condition (\ref{eq:lm-new}) can be reformulated for the spectral
measure $\sigma$ as
%
\begin{equation}
\label{eq:s-fin}
\int_{\SS} \mu(B)^\alpha\sigma(\mathrm{d}\mu)<\infty
\end{equation}
for all $B\in\sB_0$. Note that $\sigma$ can also be defined on any
other reference sphere $\SS'\subset\M$, provided each
$\mu\in\M\setminus\{0\}$ can be uniquely represented as $t\mu'$ for
some $\mu'\in\SS'$.

\begin{Th}
\label{th:finite}
Let $\sigma$ be the spectral measure of a $\sas$ random
measure $\zeta$ with L\'evy measure $\Lambda$. Then
%
\begin{equation}
\label{eq:sd}
L_\zeta[h]=\exp \biggl\{-\int_{\SS}
\langle h,\mu\rangle^\alpha\sigma(\mathrm{d}\mu) \biggr\} ,\qquad
  h\in\cz .
\end{equation}
Furthermore,
\begin{enumerate}[(ii)]
\item[(i)] $\zeta$ is \as finite if and only
if its L\'evy measure $\Lambda$ (resp.,\ spectral measure
$\sigma$)
is supported by finite measures and $\sigma(\SS)=\sigma(\M_1)$ is
finite.
\item[(ii)] The Laplace functional (\ref{eq:lf}) defines a
non-random measure if and only if $\alpha=1$. In this case
$\zeta=\overline{\mu}(\cdot)=\int_{\SS} \mu(\cdot)\sigma(\mathrm{d}\mu)$.
\end{enumerate}
\end{Th}

\begin{pf}
The representation (\ref{eq:sd}) follows from (\ref{eq:lf}), since
\begin{eqnarray*}
\int_{\M\setminus\{0\}}\bigl(1-\mathrm{e}^{-\langle h,\mu\rangle}\bigr) \Lambda(\mathrm{d}\mu)
&=&\int_{\SS} \int_0^\infty\bigl(1-\mathrm{e}^{-t \langle
h,\mu\rangle}\bigr) \theta_\alpha(\mathrm{d}t) \widehat{\sigma}(\mathrm{d}\mu)\\
&=&\Gamma(1-\alpha)\int_{\SS}\langle h,\mu\rangle^\alpha
\widehat{\sigma}(\mathrm{d}\mu),
\end{eqnarray*}
implying \eqref{eq:sd}.
\begin{longlist}[(ii)]
\item[(i)] Taking $h_n=\ind_{X_n}$, $n\geq1$, where $X_n\in\sB_0$
form a nested sequence of relatively compact sets such that
$X=\bigcup_{n} X_n$, we obtain that the Laplace transform of $\zeta(X)$
is
\[
\lim_n L_\zeta[h_n]=\exp \biggl\{-\int_{\M\setminus\{0\}}
\bigl(1-\mathrm{e}^{-\mu(X)}\bigr)\Lambda(\mathrm{d}\mu) \biggr\} ,
\]
where the integral is finite if and only if $\Lambda$ is supported
by finite measures. If this is the case, the spectral measure
$\sigma$ is defined on $\M_1$ and
\[
\lim_n L_\zeta[h_n]=\exp\{-\sigma(\M_1)\}
\]
defines a finite random variable if and only if $\sigma(\M_1)$ is
finite.

\item[(ii)] If $\zeta$ is $\sas$ with $\alpha=1$, then its values
on all relatively compact sets are deterministic, and so $\zeta$ is
a deterministic measure. Furthermore, (\ref{eq:stabmes}) clearly
holds with $\alpha=1$ for any deterministic~$\zeta$. Finally,
(\ref{eq:sd}) with $\alpha=1$ yields
\[
L_\zeta[h]=\exp \biggl\{-\int_{\SS}
\langle h,\mu\rangle\sigma(\mathrm{d}\mu) \biggr\}
=\exp \{-\langle h,\overline{\mu}\rangle \} .
\]
\end{longlist}\upqed
\end{pf}

If the spectral measure is degenerate, the corresponding $\sas$
measure has a particularly simple structure.

\begin{Th}\label{th:degen}
Assume that the spectral measure $\sigma$ of a $\sas$ measure
$\zeta$ with $0<\alpha<1$ is concentrated on a single measure so
that $\sigma=c\delta_\mu$ for some $\mu\in\SS$ and $c>0$. Then
$\zeta$ can be represented as $\zeta=c^{1/\alpha}\zeta_\alpha\mu$,
where $\zeta_\alpha$ is a positive $\sas$ random variable with
Laplace transform $\E \mathrm{e}^{-z\zeta_\alpha}=\exp\{-z^\alpha\}$.
\end{Th}
\begin{pf}
It suffices to verify that the Laplace transforms of both measures
coincide and equal $L_\zeta[h]=\exp\{-c\langle
h,\mu\rangle^\alpha\}$.
\end{pf}

\begin{example}
\label{ex:leb}
Let $X=\R^d$ and let $\sigma$ be concentrated on the Lebesgue
measure $\ell$. By Theorem \ref{th:degen}, the corresponding $\sas$
measure $\zeta$ is stationary and proportional
to $\zeta_\alpha\ell$.
\end{example}

If $X$ is compact, then all measures $\mu\in\M$ are finite and
condition (\ref{eq:s-fin}) implies that $\sigma(\M_1)$ is finite, so
that the spectral measure is defined on the family of probability
measures on $X$.

\begin{example}[(Finite phase space)]
\label{ex:finite}
In the special case of a finite $X=\{x_1,\dots,x_d\}$ the family of
probability measures becomes the unit simplex
$\Delta_d=\{x\in\R_+^d\dvtx  \sum x_i=1\}$ and a~$\sas$ random measure
$\zeta$ is nothing else but a $d$-dimensional totally skewed (or
one-sided) strictly stable random vector
$\zeta=(\zeta_1,\dots,\zeta_n)$. Its Laplace transform is given
by\looseness=-1
%
\begin{equation}
\label{eq:ltrans-finite}
\E \mathrm{e}^{-\langle h,\zeta\rangle}
=\exp \biggl\{-\int_{\Delta_d} \langle h,x\rangle^\alpha
\sigma(\mathrm{d}x) \biggr\} , \qquad   h\in\R_+^d ,
\end{equation}
where $\sigma$ is a finite measure on the unit simplex.
Alternatively, the integration can be taken
over the unit sphere. It is shown in \cite{mo07} that this Laplace
functional can be written as
%
\begin{equation}
\label{eq:supf}
\E \mathrm{e}^{-\langle h,\zeta\rangle}
=\mathrm{e}^{-H_K(h^\alpha)} ,\qquad   h\in\R_+^d ,
\end{equation}
where $h^\alpha=(h_1^\alpha,\dots,h_d^\alpha)$ and $H_K$ is the
support function of a certain convex set $K$ that appears to be a
generalisation of zonoids.
\end{example}

\begin{example}[($\sas$ random measures on $\R^d$)]
\label{ex:strm}
Fix a probability measure $\mu$ on \mbox{$X=\R^d$} and let $\sigma$ be the
image under the map $x\mapsto\mu(\cdot-x)$ of a Radon measure
$\nu$. The corresponding $\sas$ random measure $\zeta$ has the
Laplace transform
\[
L_\zeta[h]=\exp \biggl\{-\int_{\R^d} \biggl(\int_{\R^d} h(x+y)\mu
(\mathrm{d}x) \biggr)^\alpha\nu(\mathrm{d}y) \biggr\} .
\]
To ensure condition (\ref{eq:s-fin}) it is necessary to assume that
%
\begin{equation}
\label{eq:finite-Radon}
\int_{\R^d} \bigl(\mu(B-y)\bigr)^\alpha\nu(\mathrm{d}y)<\infty
\end{equation}
for all relatively compact $B$. In particular,
%
\begin{equation}\label{eq:addit}
\E \mathrm{e}^{-z\zeta(B)}=\exp \biggl\{-z^\alpha\int_{\R^d}
 \bigl(\mu(B-x) \bigr)^\alpha\nu(\mathrm{d}x) \biggr\} ,\qquad   B\in\sB_0 .
\end{equation}
A similar construction applies if $\mu$ is a general Radon measure.
Since, in
this case, it is difficult to ensure that measures $\mu(\cdot-x)$
belong to $\SS$, the spectral measure $\sigma$ is defined as the
projection onto $\SS$ of the image of $\nu$.
\end{example}

\begin{example}[(Stationary $\sas$ random measures on $\R^d$)]
\label{ex:statrm}
Consider the group $T_y$ of shifts on $\M_1$ acting as
$T_y\mu(\cdot)=\mu(\cdot-y)$. Call \textit{centroid} any measurable
map $C\dvtx \M_1\mapsto\R^d$ such that $C(T_y\mu)=C(\mu)+y$ for every
$y\in\R^d$ and $\mu\in\M_1$. For example, if $\mu^1\thru\mu^d$ are
the marginals of $\mu$, the $i$th component of $C(\mu)$ is
$C^i(\mu)=\inf\{t\dvtx \mu^i(-\infty,t]\geq1/2\}$. Let $\M_1^0$ denote
the set of probability measures with centroid $C(\mu)$ in the
origin. A $\sas$ random measure with the spectral measure supported
by $\M_1$ is stationary if and only if the spectral measure $\sigma$
is invariant with respect to $T_y$, that is, when it is decomposable into
the product $\ell\times\sigma^0$ of the $d$-dimensional Lebesgue
measure and a measure on $\M_1^0$ satisfying
\[\vspace*{2pt}
\int_{\M_1^0} \mu(B)^\alpha\sigma^0(\mathrm{d}\mu)<\infty \qquad \mbox{for
all }
B\in\sB_0 .\vspace*{2pt}
\]
Then the Laplace transform of a stationary $\sas$ measure on $\R^d$
has the following\break form:
%
\begin{equation}\vspace*{2pt}
\label{eq:lapstatsas}
L_\zeta[h]=\exp \biggl\{-\int_{\M_1^0}\int_{\R^d} \biggl(\int_{\R
^d} h(x+y)\mu(\mathrm{d}x) \biggr)^\alpha
\,\mathrm{d}y\, \sigma^0(\mathrm{d}\mu) \biggr\} .\vspace*{2pt}
\end{equation}
In particular, if $\sigma^0$ charges only one probability measure
$\mu$, we obtain $\zeta$ from Example \ref{ex:strm} with $\nu$ being
the Lebesgue measure.

More generally, take a
homogeneous of order $-(\alpha+d)$ measure
$\Lambda_0$ on $\M\setminus\{0\}$ and set $\Lambda(\mathrm{d}\mu)=\int_{\R^d}
\Lambda_0(\mathrm{d}\mu-x) \,\mathrm{d}x$. Provided \eqref{eq:lm-new} is satisfied,
$\Lambda$ is the L\'evy measure of a~stationary $\sas$ measure.
\end{example}

Note that usually in the literature, and particularly in \cite
{SamTaq94}, Section 3.3, the term $\alpha$-sta\-ble measure is
reserved for an \textit{independently scattered} measure, that is, a measure
with independent $\alpha$-stable values on disjoint sets. Our notion
is more general as the following proposition shows.

\begin{Th}
\label{th:indep}
A $\sas$ random measure with $\alpha\in(0,1)$ is independently
scattered if and only if its spectral measure $\sigma$ is supported
by the set $\{\delta_x\dvtx x\in X\}\subseteq\SS$ of Dirac
measures.\looseness=1
\end{Th}
\begin{pf}
\textit{Sufficiency}. The Laplace functional (\ref{eq:sd}) of $\zeta$
becomes
\[\vspace*{2pt}
L_\zeta[h]=\exp \biggl\{-\int_X h^\alpha(x) \tilde\sigma(\mathrm{d}x)
\biggr\} ,\vspace*{2pt}
\]
where $\tilde{\sigma}$ is the image of $\sigma$ under the map
$\delta_x\mapsto x$. It is easy to see that $L_\zeta[\sum_{i=1}^n
h_i]=\prod_{i=1}^n L_\zeta[h_i]$ for all functions $h_1\thru h_n$
with disjoint supports, meaning, in particular, that for any
disjoint Borel sets $B_1\thru B_n$ the variables $\zeta(B_i)$ with
Laplace transforms $L_\zeta[z_i\ind_{B_i}]$, $i=1\thru n$, are
independent, so that $\zeta$ is independently scattered.

\textit{Necessity}. Take two disjoint sets $B_1, B_2$ from an at most
countable base of the topology on $X$. Since $\zeta(B_1)$ and $\zeta
(B_2)$ are
independent, we have that
\[
L_\zeta[\ind_{B_1}+\ind_{B_2}]= L_\zeta[\ind_{B_1}]
L_\zeta[\ind_{B_2}] .
\]
By (\ref{eq:sd}),
\[
\int_{\SS}\bigl[\mu^\alpha(B_1) +
\mu^\alpha(B_2)-\bigl(\mu(B_1)+\mu(B_2)\bigr)^\alpha\bigr]
\sigma(\mathrm{d}\mu)=0 .
\]
Since $\alpha\in(0,1)$, the integrand is a non-negative expression
whatever $\mu(B_1), \mu(B_2)\geq0$ are. Since $\sigma$ is a
positive measure, the integral is zero only if the integrand
vanishes on the support of $\sigma$, that is, for $\sigma$-almost all
$\mu$ either $\mu(B_1)=0$ or $\mu(B_2)=0$. Since this is true for
all disjoint $B_1, B_2$ from the base, $\mu$ is concentrated at a
single point.
\end{pf}

Along the same lines it is possible to prove the following result.

\begin{Th}
The values $\zeta(B_1),\dots,\zeta(B_n)$ of a $\sas$ random measure
$\zeta$ with $\alpha\in(0,1)$ on disjoint sets
$B_1,\dots,B_n$ are independent if and only if the support of the
spectral measure $\sigma$ (or of the L\'evy measure $\Lambda$) does
not include any measure that has positive values on at least two
sets from $B_1,\dots,B_n$.
\end{Th}

\begin{example}[(Self-similar random measures)]
Recently, Vere-Jones \cite{VJ05} has introduced a wide class of self-similar
random measures on $\R^d$, that is, the measures satisfying
%
\begin{equation}
\label{eq:ssim}
\zeta(B)\deq a^{-H}\zeta(aB)
\end{equation}
for all Borel $B$, all positive $a$ and some $H,$ which is
then called the \textit{similarity index}. These measures are generally not
independently scattered although stationary independently scattered
$\sas$ measures are self-similar with index $H=1/\alpha$.

Introduce operation $D_a$ on measures $\mu\in\M$ by setting
$(D_a\mu)(B)=\mu(aB)$, $B\in\sB$, and the corresponding operation
$(\tilde{D}_a\sigma)(M)=\sigma(D_a M)$ uplifted to $\sigma$ on
measurable $M\subset\SS$. Assume that the spectral measure is
supported by $\M_1$. Since
\[
L_{a^{-H}D_a\zeta}[h(x)]= L_\zeta[a^{-H}h(a^{-1}x)] ,
\]
then writing property \eqref{eq:ssim} for the Laplace
transform \eqref{eq:sd}, we obtain
\begin{eqnarray*}
\int_\SS\langle h(x),\mu\rangle^{\alpha}\sigma(\mathrm{d}\mu)&=& \int_\SS
\langle a^{-H} h(a^{-1}x),\mu\rangle^{\alpha}\sigma(\mathrm{d}\mu)\\
&=& a^{-\alpha H}\int_{\M_1}\langle h,D_a\mu\rangle^{\alpha}\sigma
(\mathrm{d}\mu)
= a^{-\alpha H}\int_{\M_1}\langle h,\mu\rangle^{\alpha}(\tilde
{D}_{a^{-1}}\sigma)(\mathrm{d}\mu) ,
\end{eqnarray*}
where in the last equality we used the fact that that
$D_a\M_1=\M_1$. Therefore a $\sas$ random measure is self-similar
with index $H$ if and only if $\sigma=a^{-\alpha H}
\tilde{D}_{a^{-1}}\sigma$.

As in the proof of Theorem \ref{th:indep} above, in the case of
independently scattered $\sas$ measures the last identity could be
written as $\tilde{\sigma}(\mathrm{d}x)=a^{-\alpha H}
\tilde{D}_{a^{-1}}\tilde{\sigma}(\mathrm{d}x)=a^{-\alpha
H}\tilde{\sigma}(a \,\mathrm{d}x)$ for the image $\tilde{\sigma}$ of
$\sigma$ under the map $\delta_x\mapsto x$. In particular, in the
stationary case $\tilde{\sigma}$ is necessarily proportional to the
Lebesgue measure and the corresponding random measure becomes
self-similar if and only if $\alpha=1/H$.
\end{example}

The following theorem provides a LePage representation of a $\sas$
random measure.

\begin{Th}
\label{th:lp}
A random measure $\zeta$ is $\sas$ if and only if
%
\begin{equation}
\label{eq:lepagezeta}
\zeta\deq\sum_{\mu_i\in\Psi}\mu_i ,
\end{equation}
where $\Psi$ is the Poisson process on $\M\setminus\{0\}$ driven by
an intensity measure $\Lambda$ satisfying~(\ref{eq:lmreg}) and such that
$\Lambda(tA)=t^{-\alpha}\Lambda(A)$ for all $t>0$ and any measurable
$A$. In this case $\Lambda$ is exactly the L\'evy measure of
$\zeta$. Convergence of the series in \eqref{eq:lepagezeta} is
in the sense of the vague convergence of measures.

If the spectral measure $\sigma$ corresponding to $\Lambda$
satisfies $c=\sigma(\SS)<\infty$, then
%
\begin{equation}
\label{eq:lp}
\zeta\deq b\sum_{k=1}^\infty
\gamma_k^{-1/\alpha} \eps_k ,\qquad
b= \biggl(\frac{c}{\Gamma(1-\alpha)} \biggr)^{1/\alpha},
\end{equation}
where $\eps_1,\eps_2,\ldots$ are \iid random measures with
distribution $c^{-1}\sigma$ and $\gamma_k=\xi_1+\cdots+\xi_k$,
$k\geq1$, for a sequence of independent exponentially distributed
random variables $\xi_k$ with mean one.
\end{Th}
\begin{pf}
Consider a continuous function $h\in\cz$ and define a map
$\mu\mapsto\langle h,\mu\rangle$ from $\M$ to $\R_+=[0,\infty)$.
The image of $\Psi$ under this map becomes a Poisson process
$\{x_i,  i\geq1\}$ on $\R_+$ with intensity measure $\theta$ that
satisfies
\[
\theta([r,\infty))=\Lambda(\{\mu\dvtx  \langle h,\mu\rangle\geq r\}
) .
\]
Since $1-\mathrm{e}^{-x}\geq a\one_{x\geq r}$ for any $r>0$ and some constant
$a>0$, Condition (\ref{eq:lmreg}) implies that
$\theta([r,\infty))<\infty$. Then the homogeneity property of
$\Lambda$
yields that $\theta([r,\infty))=cr^{-\alpha}$ for all $r>0$ and some
constant $c>0$.

It is well known (see \cite{SamTaq94,davmolzuy08})
that for such intensity measure $\theta$ the sum $\sum_i x_i$
converges almost surely. Thus, $\sum_{\mu_i\in\Psi}\langle
h,\mu_i\rangle$ converges for each continuous $h\in\cz$, so that the
series (\ref{eq:lepagezeta}) converges in the vague topology to a
random measure see \cite{res87}, equation~(3.14).

By Resnick \cite{res87}, Proposition 3.19, a sequence of random measures converges
weakly if and only if the values of their Laplace functionals on any
continuous function converge. This is seen by noticing that the
Laplace functional (\ref{eq:lf}) of $\zeta$ coincides with the
Laplace functional of the right-hand side of (\ref{eq:lepagezeta}),
where the latter can be computed as the probability generating
functional (p.g.fl.) of a Poisson process,
see (\ref{eq:pgflpois}).

Finally, \eqref{eq:lp} follows from the polar decomposition for
$\Lambda$.
\end{pf}

\section{Discrete stability for point processes}
\label{sec:stab-with-resp}

\textit{Point processes} are counting random measures, that is, random
elements with realisations in the set of locally finite counting
measures. Each counting measure $\phi$ can be represented as the sum
$\phi=\sum\delta_{x_i}$ of unit masses, where we allow for the
multiplicity of the support points $\{x_i,  i\geq1\}$. The
distribution of a point process $\Phi$ can be characterised by its
p.g.fl. defined
on functions $u\dvtx X\mapsto(0,1]$ such that $1-u\in\cz$ by means of
%
\begin{equation}\vspace*{2pt}
\label{eq:pgfldef}
G_\Phi[u]=\E\exp \{ \langle\log u,\Phi\rangle \}=
\E\prod_{x_i\in\supp\Phi} u(x_i)^{\Phi(\{x_i\})} =L_\Phi[-\log u]\vspace*{2pt}
\end{equation}
with the convention that the product is 1 if $\Phi$ is a null measure
(i.e., the realisation of the process contains no
points). The \pgfl can be extended to pointwise monotone limits of the
functions from $\cz$ at the expense of allowing for infinite and zero
values; see, for example,  \cite{DalVJon08}, Chapter 9.4.

A \textit{Poisson process} $\Pi_\mu$ with intensity measure $\mu$ is
characterised by the property that for any finite collection of
disjoint sets $B_1\thru B_n\in\sB_0$ the variables $\Pi_\mu
(B_1)\thru
\Pi_\mu(B_n)$ are mutually independent Poisson distributed random
variables with means $\mu(B_1)\thru\mu(B_n)$. The \pgfl of a Poisson
process $\Pi_\mu$ is given by
%
\begin{equation}\vspace*{2pt}
\label{eq:pgflpois}
G_{\Pi_\mu}[u]=\exp \{ - \langle1-u,\mu\rangle \} .\vspace*{2pt}
\end{equation}
If the intensity measure itself is a random measure $\zeta$, then the
obtained (doubly
stochastic) point process is called a \textit{Cox process}. Its \pgfl\
is given by
%
\begin{equation}
\label{eq:coxpgfl}\vspace*{2pt}
G_{\Pi_\zeta}[u]=\E\exp \{ - \langle
1-u,\zeta\rangle \}=L_\zeta[1-u] .\vspace*{2pt}
\end{equation}
Note that the Cox process is stationary if and only if the random
measure $\zeta$ is stationary.\looseness=1

Addition of counting measures is well defined and leads to the
definition of the superposition operation for point processes.
However, multiplication of the counting measure by positive numbers
cannot be defined by multiplying its values -- they no longer remain
integers for arbitrary multiplication factors. In what follows we
define a stochastic multiplication operation that corresponds to the
thinning operation for point processes. Namely, each unit mass
$\delta_{x_i}$ in the representation of the counting measure
$\phi=\sum_i \delta_{x_i}$ is removed with probability $1-t$ and
retained with probability $t$ independently of other masses. The
resulting counting measure $t\circ\phi$ becomes random (even if $\phi$
is deterministic) and is known under the name of \textit{independent
thinning} in the point process literature; see, for example,  \cite
{DalVJon08}, Chapter 11.3, or
\cite{mkm}, 
Chapter 7, where relations to cluster and Cox processes were established.

To complement this pathwise description, it is possible to define the
thinning operation on probability distributions of point
processes. Namely, the thinned process $t\circ\Phi$ has the \pgfl
%
\begin{equation}\vspace*{2pt}
\label{eq:pgflthin}
G_{t\circ\Phi}[u]=G_\Phi[tu+1-t]=G_\Phi[1-t(1-u)] ;\vspace*{2pt}
\end{equation}
see, for example, \cite{DalVJon08}, page 155. From this, it is easy to establish
the following properties of the thinning operation.

\begin{Th}\label{th:thinprop}
The thinning operation $\circ$ is associative, commutative and
distributive with respect to superposition of point processes, that is,
%
\[
 t_1\circ(t_2\circ\Phi)\deq(t_1 t_2)\circ\Phi\deq
t_2\circ(t_1\circ\Phi)
\]
and
\[
t\circ(\Phi+\Phi')\deq(t\circ\Phi) + (t\circ\Phi')
\]
for any $t, t_1, t_2\in[0,1]$ and any independent point processes
$\Phi$ and $\Phi'$. For disjoint Borel sets $B_1$ and $B_2$, random
variables $(t\circ
\Phi)(B_1)$ and $(t\circ\Phi)(B_2)$ are conditionally independent
given $\Phi$ and there exists a coupling of $\Phi$ and $t\circ\Phi$
(described above) such that $(t\circ\Phi)(B)\leq\Phi(B)$ almost
surely for any $B\in\sB$.
\end{Th}
%


If the phase space $X$ consists of one point, the point process $\Phi$
becomes a non-negative integer random variable $\Phi(X)$. Since
\eqref{eq:pgflthin} turns into \eqref{eq:tphi}, the thinning
operation becomes the classical discrete multiplication operation
acting on positive integer random variables. Keeping this in mind, we
can define the notion of discrete stability for point processes.

\begin{definition}
\label{def:das}
A point process $\Phi$ (or its probability distribution) is called
\textit{discrete $\alpha$-stable} or \textit{$\alpha$-sta\-ble with
respect to thinning} (notation $\das$), if for any $0\leq t \leq1$
one has
%
\begin{equation}
\label{eq:thinstabdef}
t^{1/\alpha} \circ\Phi'+(1-t)^{1/\alpha}\circ\Phi''\deq\Phi ,
\end{equation}
where $\Phi'$ and $\Phi''$ are independent copies of $\Phi$.
\end{definition}

Definition \ref{def:das} and \eqref{eq:pgflthin} yield that $\Phi$ is
$\das$ if and only if its \pgfl possesses the property:
%
\begin{equation}
\label{eq:daspgfl}
G_\Phi[1-t^{1/\alpha}(1-u)] G_\Phi[1-(1-t)^{1/\alpha}(1-u)]
=G_\Phi[u]
\end{equation}
for any positive function $u$ such that $1-u\in\cz$ and all $0<t<1$.

Relation \eqref{eq:thinstabdef} implies that $\Phi(B)$ is a discrete
$\alpha$-stable random variable for any \mbox{$B\in\sB$}. Hence $\das$ point
processes exist only for $\alpha\in(0,1]$ due to the already mentioned
result~\eqref{eq:lawdas}.
Now we come to the main result of this section.
\begin{Th}
\label{th:coxpros}
A point process $\Phi$ is $\das$ if and only if it is a Cox process
$\Pi_\zeta$ with a~$\sas$ intensity measure $\zeta$.
\end{Th}
\begin{pf}
\textit{Sufficiency}. If $\zeta$ is a $\sas$ random measure, then the
corresponding Cox process~$\Pi_\zeta$ has the \pgfl
%
\begin{equation}
\label{eq:pg-pp}
G_{\Pi_\zeta}[u]=L_\zeta[1-u]=
\exp \biggl\{-\int_{\M\setminus\{0\}}
\bigl(1-\mathrm{e}^{-\langle1-u,\mu\rangle}\bigr) \Lambda(\mathrm{d}\mu) \biggr\} .
\end{equation}
Thus,
\begin{eqnarray*}
G_\Phi[1-t(1-u)]&=&\exp \biggl\{-\int_{\M\setminus\{0\}}
\bigl(1-\mathrm{e}^{-\langle1-u,t\mu\rangle}\bigr) \Lambda(\mathrm{d}\mu) \biggr\}\\
&=&\exp \biggl\{-t^\alpha\int_{\M\setminus\{0\}}
\bigl(1-\mathrm{e}^{-\langle1-u,\mu\rangle}\bigr) \Lambda(\mathrm{d}\mu) \biggr\}
\end{eqnarray*}
by the homogeneity of $\Lambda$. Then (\ref{eq:daspgfl}) clearly
holds.

\textit{Necessity}.
By iterating \eqref{eq:thinstabdef} $m$ times we arrive at
\[
m^{-1/\alpha}\circ\Phi_1+\cdots+m^{-1/\alpha}\circ\Phi_m \deq
\Phi
\]
for \iid$\Phi,\Phi_1,\dots,\Phi_m$, implying that $\das$ point
processes are necessarily infinitely divisible and
%
\begin{equation}
\label{eq:infdiv}
 \bigl(G_\Phi[1-m^{-1/\alpha}(1-u)] \bigr)^m=G_\Phi[u]
\end{equation}
for all $m\geq2$.

The crucial step of the proof aims to show that the functional
%
\begin{equation}
\label{eq:lg}
L[u]=G_\Phi[1-u] ,\qquad   u\in\cz ,
\end{equation}
is the Laplace functional of a $\sas$ random measure $\zeta$. While
the functional $L$ on the left-hand side should be defined on all
(bounded) functions with compact supports, it is apparent that the
\pgfl$G_\Phi$ on the right-hand side of (\ref{eq:lg}) may not be
well defined on $1-u$. Indeed, in contrast to (\ref{eq:pgfldef}),
$1-u$ does not necessarily take values from the unit interval. To
overcome this difficulty, we employ (\ref{eq:infdiv}) and define
%
\begin{equation}
\label{eq:ml}
L[u]=(G_\Phi[1-m^{-1/\alpha}u])^m .
\end{equation}
Since $u\in\cz$, for sufficiently large $m$ the function
$1-m^{-1/\alpha}u$ does take values in $[0,1]$ and equals $1$ outside a
compact set. Since (\ref{eq:ml}) holds for all $m$, it is possible
to pass to the limit as $m\ti$ to see that
\[
L[u]=\exp \Bigl\{-\lim_{m\ti} m(1-G_\Phi[1-m^{-1/\alpha}u]) \Bigr\} .
\]
By the Schoenberg theorem (see, e.g., \cite{BergChriRes84}, Theorem 3.2.2) $L$ is
positive definite if $1-G_\Phi[1-m^{-1/\alpha}u]$ is negative
definite, that is,
\[
\sum_{i,j=1}^n c_ic_j \bigl(1-G_\Phi[1-m^{-1/\alpha}(u_i+u_j)]\bigr)\leq0
\]
for all $n\geq2$, $u_1,\dots,u_n\in\cz$ and $c_1,\dots,c_n$ with
$\sum c_i=0$. In view of the latter condition, we need to show the
inequality
\[
\sum_{i,j=1}^n c_ic_j G_\Phi[1-m^{-1/\alpha}(u_i+u_j)]\geq0 .
\]
Note that
\[
\lim_{m\ti} \frac{1-m^{-1/\alpha}u}{\mathrm{e}^{-m^{-1/\alpha} u}}=1 .
\]
Denote $v_i=\mathrm{e}^{-m^{-1/\alpha} u_i}$, $i=1,\dots,n$. Referring to
the continuity of the \pgfl$G_\Phi$, it suffices to check that
\[
\sum_{i,j=1}^n c_ic_j G_\Phi[v_iv_j]\geq0 ,
\]
which is exactly the positive definiteness of $G_\Phi$.

Thus, $L_\zeta[\sum_{i=1}^k t_ih_i]$ as a function of
$t_1,\dots,t_k\geq0$ is the Laplace transform of a random vector.
Arguing as in the proof of sufficiency in
Theorem \ref{thr:gen-stas}, it is easy to check the continuity of
$L$, so that $L$ is indeed the Laplace functional of a random
measure $\zeta$. Condition (\ref{eq:daspgfl}) rewritten for $L$
means that $\zeta$ is $\sas$.
\end{pf}

By Theorem \ref{th:finite}, a $\sas$ random measure $\zeta$ has
the characteristic exponent $\alpha=1$ if and only if $\zeta$ is
deterministic. Respectively, a $\das$ processes with $\alpha=1$ is
a Poisson process driven by (non-random) intensity measure $\zeta$.

Using decomposition $\Lambda=\hat{\sigma}\times\theta_\alpha$, the
first identity in \eqref{eq:pg-pp} and \eqref{eq:sd}, we obtain the
following result.

\begin{Cor}[(Spectral representation)]
\label{cor:spect}
A point process $\Phi$ is $\das$ with $\alpha\in(0,1]$ if and only
if its \pgfl has the form
%
\begin{equation}
\label{eq:pg-pp-spec}
G_{\Phi}[u]=\exp \biggl\{-\int_{\SS}
\langle1-u,\mu\rangle^\alpha\sigma(\mathrm{d}\mu) \biggr\} ,\qquad
 1-u\in\cz,
\end{equation}
for some locally finite spectral measure $\sigma$ on $\SS$ that
satisfies (\ref{eq:s-fin}).
\end{Cor}

The number of points $\Phi(B)$ of a $\das$
process $\Phi$ in a relatively compact Borel set $B$ has the \pgf\
%
\begin{equation}\label{eq:pgfnop}
\E s^{\Phi(B)}=\exp \biggl\{-(1-s)^\alpha\int_{\SS}
\mu(B)^\alpha\sigma(\mathrm{d}\mu) \biggr\}
\end{equation}
and so $\Phi(B)$ is either zero \as or has infinite expectation for
$0<\alpha<1$, while $\E\Phi(B)$ is finite in the Poisson case
$\alpha=1$. Furthermore, the avoidance probability is given by
\[
\P\{\Phi(B)=0\}=\exp \biggl\{-\int_{\SS}\mu(B)^\alpha\sigma(\mathrm{d}\mu
) \biggr\} ,
\]
so that (\ref{eq:s-fin}) guarantees that the avoidance probabilities
are positive. In other words, a~$\das$ point process does not have
fixed points.

Theorems \ref{th:lp} and \ref{th:coxpros} immediately imply the
following result.

\begin{Cor}[(LePage representation of a $\das$ process)]
\label{cor:lpdas}
A $\das$ process $\Phi$ with L\'evy measure $\Lambda$ can be
represented as a Cox process:
%
\begin{equation}
\label{eq:dlep1}
\Phi\deq\sum_{\mu_i\in\Psi} \Pi_{\mu_i} ,
\end{equation}
where $\Psi$ is a Poisson
process on $\M\setminus\{0\}$ with intensity measure $\Lambda$. In
particular, if its spectral measure $\sigma$ is finite, then
%
\begin{equation}
\label{eq:dlep2}
\Phi\deq\sum_{k=1}^\infty\Pi_{b\gamma_k^{-1/\alpha}\eps_k} ,
\end{equation}
where $\epsilon_k$, $\gamma_k$ and $b$ are defined in
Theorem \ref{th:lp} and the sum in (\ref{eq:dlep2}) almost surely
contains only finitely many terms.
\end{Cor}
\begin{pf}
By Theorem \ref{th:coxpros}, $\Phi$ is a Cox process $\Pi_\zeta$,
where $\zeta$ is representable as \eqref{eq:lepagezeta}.
Conditioned on a realisation of $\Psi$, we have that
$\Pi_{\zeta}=\sum_i \Pi_{\mu_i}$ by the superposition property of a
Poisson process. Now the statement easily follows.

The finiteness of the sum in \eqref{eq:dlep2} follows from the
Borel--Cantelli lemma, since
\[
\P\{\Pi_{b\gamma_k^{-1/\alpha}\eps_k}(X)>0 \cond  \gamma_k\}
=1-\exp\{-b\gamma_k^{-1/\alpha}\} =\mathcal{O}(k^{-1/\alpha})
\]
for almost all realisations of $\{\gamma_k\}$.
\end{pf}



\begin{Cor}[(see \cite{Dev93})]
\label{cor:dev}
A random variable $\xi$ with non-negative integer values is discrete
$\alpha$-stable if and only if $\xi$ is a mixture of
Poisson laws with parameter given by a~positive strictly stable random
variable. The \pgf of $\xi$ is given by
\[
\E s^\xi= \mathrm{e}^{-c(1-s)^\alpha} ,\qquad   u\in(0,1] ,
\]
for some $c>0$.
\end{Cor}

\begin{example}[(Discrete stable vectors, \cf Example \ref{ex:finite})]
\label{ex:fstsp}
A point process $\Phi$ on a~finite space $X$ can be described by a
random vector $(\xi_1,\dots,\xi_d)$ of dimension $d$ with
non-negative integer components. If $\Phi$ is $\das$ then
$(\xi_1,\dots,\xi_d)$ are said to form a $\das$ random vector. For
instance, if $\Phi$ is a $\das$ point process in any space $X$, then
the point counts $\Phi(B_1),\dots,\Phi(B_n)$ form a $\das$ random
vector for $B_1,\dots,B_n\in\sB_0$.

Theorem \ref{th:coxpros} implies that $\xi=(\xi_1\thru\xi_d)$ is
$\das$ if and only if its components are mixtures of Poisson random variables
with parameters $\zeta=(\zeta_1\thru\zeta_d)$ and are conditionally
independent given $\zeta$, where $\zeta$ is a strictly stable
non-negative random vector with the Laplace transform
(\ref{eq:ltrans-finite}). Notice that, in general, the components of
$\xi$ are \textit{dependent} $\das$ random variables unless the
spectral measure $\sigma$ of $\zeta$ is supported by the vertices of
the simplex $\Delta_d$ only. In view of (\ref{eq:supf}), the \pgf of
$\xi$ can be represented as
\[
\E\prod_{i=1}^d s_i^{\xi_i}
=\exp \biggl\{-\int_{\Delta_d}\langle
\one-s,\mu\rangle^\alpha\sigma(\mathrm{d}\mu) \biggr\}
=\mathrm{e}^{-H_K((\one-s)^\alpha)} ,
\]
where $s=(s_1,\dots,s_d)$ and $\one=(1,\dots,1)$. Thus, the
distribution of $\xi$ is determined by the values of the support
function of $H_K$ and its derivatives at $\one$.
\end{example}

\begin{example}[(Mixed Poisson process with stable intensity, \cf
Example \ref{ex:leb})]
\label{ex:lm}
Let $X=\R^d$ and let $\sigma$ attach a mass $c$ to the measure
$a\ell$, where $\ell$ is the Lebesgue measure
on $\R^d$ and $a>0$ is chosen so that $a\ell\in\SS$. Then
$G_\Phi[u]=\exp\{-ca^\alpha\langle
1-u,\ell\rangle^\alpha\}$, which is the \pgfl of a stationary Cox
process with a $\sas$ density, that is, a stationary Poisson process in
$\R^d$ driven by the random intensity measure
$ac^{1/\alpha}\zeta_\alpha\ell$, where $\zeta_\alpha$ is a positive
$\sas$ random variable with Laplace transform $\E
\mathrm{e}^{-z\zeta_\alpha}=\mathrm{e}^{-z^\alpha}$. This type of Cox process is also
known as a \textit{mixed Poisson process}. In particular, if
$c=a^{-1/\alpha}$, then $\E s^{-\Phi(B)}
=\mathrm{e}^{-(1-s)^\alpha\ell(B)^\alpha}$, so that
$\P\{\Phi(B)=0\}=\mathrm{e}^{-\ell(B)^\alpha}$ for all $B\in\sB_0$. Note that
this point process has an infinite intensity measure.
\end{example}



Discrete $\alpha$-stable point processes appear naturally as
limits for thinned superpositions of point processes. Let $\Psi$ be a
point process on $X$ and let $S_n=\Psi_1+\cdots+\Psi_n$ be the sum of
independent copies of $\Psi$. The \pgfl\ of the thinned point process
$a_n\circ S_n$ is given by\looseness=1
\[
G_{a_n\circ S_n}[1-h]=
 (G_\Psi[1-a_n h] )^n ,
\]
where $a_n\to0$ is a certain normalising sequence. On the other hand, the
superposition can be normalised by scaling its values as $a_n S_n$, so
that the result of scaling is a random measure, but no longer
counting. The following basic result establishes a correspondence
between the convergence of the thinned and the scaled superpositions.

\begin{Th}\label{th:clt}
Let $\{S_n$, $n\geq1\}$, be a sequence of point processes. Then for
some sequence
$\{a_n\}$, $a_n S_n$ weakly converges to a non-trivial random
measure that is necessarily $\sas$ with a spectral measure $\sigma$
if and only if $a_n\circ S_n$ weakly converges to a non-trivial
point process that is necessarily $\das$ with the same spectral
measure $\sigma$.

If $S_n$ is the sum of $n$ independent copies of a process $\Psi$,
the measure $\sigma$ can be defined by its finite-dimensional
distributions: if $\xi=(\Psi(B_1),\dots,\Psi(B_d))$ for
$B_1,\dots,B_d\in\sB_0$, then
%
\begin{equation}
\label{eq:limsig}
\sigma\bigl(\{\mu\dvtx (\mu_1,\dots,\mu_d)\in A\}\bigr)
=\Gamma(1-\alpha)\lim_{n\ti}n \P\{\xi/\|\xi\|_1\in A ;
\|\xi\|_1>a_n\}
\end{equation}
for all measurable $A$ from the unit $\ell_1$-sphere
$\{x\in\R_+^d\dvtx  \|x\|_1=1\}$.
\end{Th}
\begin{pf}
The equivalence of the convergence statements is established in
\cite{Kal83}, Theorem~8.4; see also \cite{DalVJon08}, Theorem 11.3.III.

Finally, \eqref{eq:limsig} is the standard condition for $a_n S_n$
to have a limit that is valid in a~much more general setting than for
random measures; see \cite{davmolzuy08}, Theorem 4.3, and
\cite{aragin}. The gamma factor stems from the particular
normalisation of the spectral measure adopted in
Sec\-tion~\ref{sec:strictly-stable-rand}.
\end{pf}

If a point process $\Psi$ has a finite intensity measure, then the
strong law of large numbers applied to its values on any relatively
compact set implies that the limit of $n^{-1}S_n$ is a~%
deterministic measure, so that the limiting $\das$ process is, in fact,
Poisson. Therefore a $\das$ limit with $\alpha\in(0,1)$ is only
possible for a point process $\Psi$ with infinite intensity. If $X$
consists of one point, then we recover the result of \cite{chrsch98}
concerning a relationship between discrete stable random variables and
conventional stable laws.




\begin{remark}
Note that not all point processes can give a non-trivial limit in
the above schemes. Consider, for instance, a point process $\Psi$
defined on $X=\{1,2,\dots\}$ with the point multiplicities $\Psi(i)$
being independent discrete stable random variables with characteristic exponents
$\alpha_i=1/2+1/(2i)$, $i=1,2,\dots.$ Then $n^{-1/\alpha}\circ S_n$
for $\alpha\leq1/2$ gives null process as the limit, while for
larger $\alpha$ the limiting process is infinite on all sufficiently
large $i$.
\end{remark}

\section{Sibuya point process and cluster representation}
\label{sec:poiss-mixt-repr}

Recall that a general cluster process is defined by means of a
\textit{centre point process} $N_c$ in some phase space $Y$ and a
countable family of independent \textit{daughter point processes}
$N(\cdot|y)$ in a phase space $X$ indexed by the points of $Y$. Their
superposition in $X$ defines a \textit{cluster process}. The \pgfl
$G[h]$ of a cluster process is then a composition $G_c[G_d[h|\cdot]]$
of the centre and the daughter processes; see, for example,  \cite
{DalVJon02}, Proposition 6.3.II.

Noting \eqref{eq:pgflpois}, the form of \pgfl(\ref{eq:pg-pp})
suggests that a $\das$ process can be regarded as a cluster process
with centre processes being Poisson with intensity measure $\Lambda$
in $Y=\M\setminus\{0\}$ and daughter processes being also Poisson
parametrised by their intensity measure $\mu\in\supp\Lambda$. We will
embark on exploration of the general case in the next section, but
here we concentrate on the case when measure $\Lambda$ charges only
finite measures and give an alternative explicit cluster characterisation
of such $\das$ processes.

\begin{definition}
Let $\mu$ be a probability measure on $X$. A point process
$\Upsilon$ on $X$ defined by the p.g.fl.
%
\begin{equation}
\label{eq:sibpgfl}
G_\Upsilon[u]=G_{\Upsilon(\mu)}[u]=1-\langle
1-u,\mu\rangle^\alpha
\end{equation}
is called the \textit{Sibuya point process} with \textit{exponent}
$\alpha$ and \textit{parameter measure} $\mu$. Its distribution is
denoted by $\Sib(\alpha,\mu)$.
\end{definition}

If $X$ consists of one point, then the point multiplicity is a random
variable $\eta$ with the~p.g.f.
\[
\E z^\eta=1-(1-z)^\alpha , \qquad  z\in(0,1] .
\]
In this case we say that $\eta$ has the Sibuya distribution and denote
it by
$\Sib(\alpha)$. The Sibuya distribution corresponds to the number of
trials to get the first success in a series of Bernoulli trials with
probability of success in the $k$th trial being $\alpha/k$; see
also \cite{Dev92,Dev93} for efficient algorithms of its
simulation.

In particular, (\ref{eq:sibpgfl}) implies that $\Upsilon(B)$ with
$B\in\sB_0$ has the \pgf
\[
\E u^{\Upsilon(B)}=1-\mu^\alpha(B)(1-u)^\alpha .
\]
Note that $\Upsilon(B)$ has infinite expectation if $\mu(B)$ does not vanish.
Furthermore,
\begin{eqnarray*}
\P\{\Upsilon(B)=0\}&=&1-\mu^\alpha(B) ,\\
\P\{\Upsilon(B)=1\}&=&\alpha\mu^\alpha(B)=q_1(\alpha) \mu^\alpha
(B) ,\\
\P\{\Upsilon(X)=n\}&=&(1-\alpha) \biggl(1-\frac{\alpha}{2} \biggr)
\cdots \biggl(1-\frac{\alpha}{n-1} \biggr)  \frac{\alpha}{n} \mu
^\alpha(B)
\bydef q_n(\alpha) \mu^\alpha(B), \qquad  n\geq2 .
\end{eqnarray*}
Therefore, conditioned to be non-zero, $\Upsilon(B)$ has the Sibuya
distribution with parameter~$\alpha$ justifying the chosen name for
this process $\Upsilon$. In the terminology of \cite{chrsch98},
$\Upsilon(B)$ has $\mu^\alpha(B)$-\textit{scaled Sibuya distribution}.
In particular, $\Upsilon(X)$ is non-zero \as and follows $\Sib(\alpha)$
distribution.

Developing the \pgfl\eqref{eq:sibpgfl} makes it possible to get
insight into the structure of a Sibuya process:
\begin{eqnarray*}
G_\Upsilon[u] & =&\E\prod_{x_i\in\Upsilon} u(x_i)=1-(1-\langle
u,\mu\rangle)^\alpha
= \sum_{n=1}^\infty q_n(\alpha) \langle u,\mu\rangle^n \\
& =& \sum_{n=1}^\infty q_n(\alpha)
\int_{X^n} u(x_1)\cdots u(x_n) \mu(\mathrm{d}x_1)\cdots\mu(\mathrm{d}x_n) .
\end{eqnarray*}
Therefore, as we have already seen, the total number of points of
$\Upsilon$ follows Sibuya distribution and, given this total number,
the points are independently identically distributed according to the
distribution $\mu$. This also justifies the fact that
(\ref{eq:sibpgfl}) indeed is a \pgfl of a point process constructed
this way. This type of processes is called \textit{purely random}
in~\cite{mkm}, page 104.

Assume now that the L\'evy measure $\Lambda$ of a $\sas$ random measure
$\zeta$ is supported by finite measures. Then (\ref{eq:sd}) holds with
the spectral measure $\sigma$ defined on $\M_1$, so that the \pgfl of
the corresponding $\das$ process takes the form
%
\begin{equation}
\label{eq:pois-sib}
G_\Phi[u]=\exp \biggl\{\int_{\M_1}
\bigl(G_{\Upsilon(\mu)}[u]-1\bigr) \sigma(\mathrm{d}\mu) \biggr\}
\end{equation}
with $G_{\Upsilon(\mu)}[u]$ given by \eqref{eq:sibpgfl}. Thus we have
shown the following result.

\begin{Th}
\label{th:pois-sib}
A $\das$ point process $\Phi$ with the L\'evy measure supported by
finite measures (equivalently, with a spectral measure $\sigma$
supported by $\M_1$) can be represented as a cluster process with
a Poisson centre process on $\M_1$ driven by intensity measure
$\sigma$ and daughter processes being Sibuya processes
$\Sib(\alpha,\mu), \mu\in\M_1$. Its \pgfl is given
by \eqref{eq:pois-sib}.
\end{Th}

As a by-product, we have established the following fact: Since Sibuya
processes are finite with probability 1, the cluster process is
finite or infinite depending on the finiteness of the Poisson
processes of centres.

\begin{Cor}
\label{cor:finitedas}
A $\das$ point process is finite if and only if its spectral measure
$\sigma$ is finite and is supported by finite measures.
\end{Cor}

If $\alpha=1$, then the Sibuya process consists of one point
distributed according to $\mu$ and the cluster process represents a
Poisson process with intensity measure
$\overline{\mu}(\cdot)=\int_{\M_1} \mu(\cdot)\sigma(\mathrm{d}\mu)$,
which is
clearly discrete 1-stable.

By Corollary \ref{cor:dev}, $\xi$ is discrete $\alpha$-stable if and
only if $\xi$ can be represented as a sum of a Poisson number of
independent Sibuya distributed random variables, which is proved for
discrete stable laws in \cite{Dev93}.

Later on we make use of Theorem \ref{th:pois-sib} for the case of an
infinite countable phase space~$X$. The finite case is considered
below.\vspace{-3pt}

\begin{definition}
Let $\mu=(\mu_1\thru\mu_d)$ be a $d$-dimensional probability
distribution. A
random vector $\Upsilon$ has \textit{multivariate Sibuya distribution}
with parameter measure $\mu$ and exponent $\alpha\in(0,1]$ if its
\pgf has
the following form:\vspace{-5pt}
%
\begin{equation}
\label{eq:mvsib}
\E\prod_{n=1}^d z_n^{\Upsilon_n}=1- \Biggl(\sum_{n=1}^d
(1-z_n)\mu_n \Biggr)^\alpha .\vspace{-5pt}
\end{equation}
\end{definition}

If $d=1$ and $\mu=1$, then the multivariate Sibuya distribution
becomes the ordinary Sibuya distribution with exponent $\alpha$. For
$d\geq2$, the marginals $\Upsilon_n$ of a multivariate Sibuya vector
$\Upsilon=(\Upsilon_1\thru\Upsilon_d)$ have the \pgf given by\vspace{-5pt}
\[
\E z_n^{\Upsilon_n}=1-\mu_n(1-z_n)^\alpha,\qquad  n=1\thru d .\vspace{-3pt}
\]
Thus $\Upsilon_n$ takes value 0 with probability $1-\mu_n$ but,
conditional on being non-zero, it is $\Sib(\alpha)$-distributed.\vspace{-3pt}

\begin{example}[($\das$ random vectors)]
\label{ex:das-vectors}
Let $\xi=(\xi_1,\dots,\xi_d)$ be a $\das$ vector.
By Theorem \ref{th:pois-sib}, it can be represented as a sum of
multivariate Sibuya $\Sib(\alpha,\mu_i)$ vectors, where~$\mu_i$ are
chosen from a finite Poisson point process on $\Delta_d$ with some
intensity measure $\sigma$.\vspace{-3pt}
\end{example}

\begin{example}[(Stationary $\das$ processes)]
\label{ex:gauss}
As in Example \ref{ex:strm}, let $X=\R^d$ with $\nu$ being
proportional to the Lebesgue measure and $\mu$ being the uniform
distribution on the unit ball centred at the origin. Then the
corresponding $\das$ process is a cluster process that can be
described by the following procedure. First, the Boolean model of
unit balls in $\R^d$ is generated with the centres following the
Poisson point process with intensity measure~$\nu$; see \cite{skm}.
Then a $\Sib(\alpha)$ number of points is thrown into each such ball
uniformly and independently from the other balls. The set of thus
generated points is a realisation of the $\das$ process.

In a more general model, the uniform distribution on the unit ball
can be replaced by any probability distribution kernel $P(\mathrm{d}y,x)$,
$x\in X$. A typical realisation of such a~process in $\R^2$ with
$\nu$ being proportional to the Lebesgue measure on $[0,1]^2$ and
$P(\mathrm{d}y,x)$ being the Gaussian distribution centred at $x$ with \iid\
components of a certain variance~$s^2$ is presented in
Figure \ref{fig:das}. The avoidance probabilities of the obtained
point process are given by\vspace{-5pt}
\[
\P\{\Phi(B)=0\}=\exp \biggl\{-\int_{\R^d}
 \bigl(\mu_0\bigl(s^{-1}(B+x)\bigr) \bigr)^\alpha \,\mathrm{d}x \biggr\} ,\vspace{-3pt}
\]
where $\mu_0$ is the standard Gaussian measure in $\R^2$.

\begin{figure}

\includegraphics{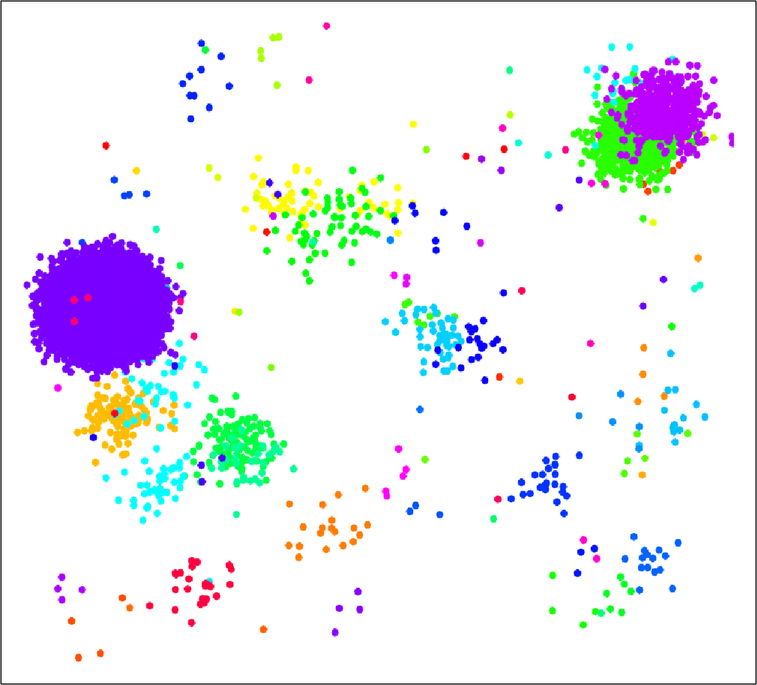}
\vspace*{-3pt}
\caption{A realisation of a $\das$ stationary process with the
Gaussian kernel. Different shades of gray correspond to
different clusters.}
\label{fig:das}\vspace*{-5pt}
\end{figure}

Figure \ref{fig:das} shows that realisations of such a process are
highly irregular. Many clusters with small numbers of points appear
alongside huge clusters, so that the resulting point process
has an infinite intensity measure. In view of this, $\das$ point
processes can help to model point patterns with point counts
exhibiting heavy-tail behaviour.
\end{example}

\section{Regular and singular $\das$ processes}
\label{sec:infinite-das-point}

We have shown in the proof of Theorem \ref{th:coxpros} that a $\das$
point process is necessarily infinitely divisible. It is known
(\seg \cite{DalVJon08}, Theorem 10.2.V) that the \pgfl of an infinitely
divisible point process admits the following representation
%
\begin{equation}
\label{eq:klm}
G_\Phi[u]=\exp \biggl\{ \int_{\sN_0}  \bigl[\mathrm{e}^{\langle\log
u,\phi\rangle} -1  \bigr] \Q(\mathrm{d}\phi) \biggr\} ,
\end{equation}
where $\sN_0$ is the space of locally finite non-empty point
configurations on $X$ and $\Q$ is a~locally finite measure on it
satisfying
%
\begin{equation}
\label{eq:klmreg}
\Q \bigl(\{\phi\in\sN_0\dvtx \phi(B)>0\} \bigr)<\infty \qquad \mbox{for
all }
B\in\sB_0(X) .
\end{equation}
This measure $\Q$ is called the \textit{KLM measure} (\textit{canonical
measure} in the terminology of~\cite{mkm}) 
and it is uniquely defined for a given infinitely divisible point
process $\Phi$. Such point process $\Phi$ is called \textit{regular} if
its KLM measure is supported by the set $\{\phi\in\sN_0\dvtx
\phi(X)<\infty\}$; otherwise it is called \textit{singular}.

It is easy to see that the expression \eqref{eq:pg-pp} for the \pgfl of
$\das$ process combined with~\eqref{eq:pgflpois} has the form
\eqref{eq:klm} with
%
\begin{equation}\label{eq:q}
\Q(\cdot)=\int_{\M\setminus\{0\}} \P_\mu(\cdot) \Lambda(\mathrm{d}\mu
) ,
\end{equation}
where $\P_\mu$ is the distribution of a Poisson point process driven
by intensity measure~$\mu$. Moreover, the requirement
\eqref{eq:klmreg} is exactly the property \eqref{eq:lmreg} of the
L\'evy measure with~$h=\ind_B$. The Sibuya cluster representation of
$\das$ processes from Section \ref{sec:poiss-mixt-repr} arises from~%
\eqref{eq:q} by integrating out the radial component in the polar
decomposition $\M_1\times\R_+$, where $\Lambda$ is concentrated in the
regular case.

The following decomposition result is inherited from
the corresponding decomposition known for infinitely divisible
processes; see, for example, \cite{DalVJon08}, Theorem 10.2.VII. Let
$\M_f=\{\mu\in\M\setminus\{0\}\dvtx  \mu(X)<\infty\}$ and
$\M_\infty=\{\mu\in\M\dvtx  \mu(X)=\infty\}$.

\begin{Th}
\label{th:decomp}
A $\das$ point process with L\'evy measure $\Lambda$ and spectral
measure $\sigma$ can be represented as a superposition of two
independent infinitely divisible processes: a regular and a singular
one. The regular process is a cluster process having a \pgfl given
by \eqref{eq:pois-sib} with spectral measure
$\sigma|_{\M_1}=\sigma(\cdot\ind_{\M_1})$ and the singular process
is a cluster process with a \pgfl given by \eqref{eq:pg-pp}
corresponding to L\'evy measure $\Lambda|_{\M_\infty}$ (or by
\eqref{eq:pg-pp-spec} with spectral measure
$\sigma|_{\SS\setminus\M_1}$).
\end{Th}
\begin{pf}
Represent $\Lambda$ as the sum $\Lambda|_{\M_f}+
\Lambda|_{\M_\infty}$ of two orthogonal measures: $\Lambda$
restricted on the set $\M_f$ and on $\M_\infty$. Then $\Q$ also
decomposes into two orthogonal measures: the one concentrated on
finite configurations and the one concentrated on infinite point
configurations (because the corresponding Poisson process $\Pi_\mu$
is finite or infinite with probability~1, correspondingly). Finally,
$G_\Phi[u]$ in \eqref{eq:klm} separates into the product of p.g.fl.'s corresponding to the cluster process studied in the
previous section and the second one with p.g.fl. \eqref{eq:pg-pp} with
$\Lambda$ replaced by $\Lambda|_{\M_\infty}$.
\end{pf}

Recall that a measure $\mu$ is called \textit{diffuse} if $\mu(\{x\})=0$
for any $x\in X$. Representation~\eqref{eq:q} of the KLM measure
immediately gives rise to the following results analogous to the
properties of infinitely divisible processes.
\begin{Th}\label{th:dasprops}
A $\das$ point process $\Phi$ with L\'evy measure $\Lambda$ and spectral
measure $\sigma$ is
\begin{itemize}[(iii)]
\item[(i)] Simple if and only if $\Lambda$ (resp.,
$\sigma$) is supported by diffuse measures,
\cf\cite{mkm}, Proposition~2.2.9. 
%
\item[(ii)] Independently scattered if and only if
$\sigma$ is supported by the set $\{\delta_x\dvtx x\in X\}\subset
\SS$ of Dirac measures, \cf\cite{mkm}, Proposition 2.2.13.
%
\item[(iii)] The distribution $\P_1$ of a $\das$ process
$\Phi_1$ with spectral measure $\sigma_2$ is absolutely continuous
with respect to the distribution $\P_2$ of another $\das$ process
$\Phi_2$ with spectral measure $\sigma_2$ and the same $\alpha$ if and
only if there exists a measurable set $A\subseteq
\M_\infty\cap\SS=\SS_\infty$ such that
$\sigma_1|_{\SS_\infty}=\sigma_2|_{A\cap\SS_\infty}$ and
$\sigma_1|_{\M_1}\ll\sigma_2|_{\M_1}$.
\end{itemize}
\end{Th}

If $X$ is compact, then all $\das$ processes are regular. In the
settings of Example \ref{ex:stdas}, $\Phi$ is singular if and only if
$\mu$ is infinite.

Consider the practically important case of stationary $\das$ processes.
\begin{Th}
A stationary $\das$ process $\Phi$ in $\R^d$ with spectral measure
$\sigma$ is:
\begin{itemize}[(ii)]
\item[(i)] Mixing if and only if
\[
\sigma\{\mu\dvtx \mu(B+x)\not\to0 \mbox{ as $\|x\|\ti$ for some
$B\in\sB_0$}\}=0 .
\]
\item[(ii)] Ergodic (or weak mixing) if and only if
\[
n^{-d} \int_{[-{n}/{2},{n}/{2}]^d}
\bigl(1-\mathrm{e}^{-\mu(B+x)}\bigr)\,\mathrm{d}x \to0  \qquad \mbox{as }  n\ti
\]
for all $B\in\sB_0$ and all $\mu$ from the support of $\sigma$.
\end{itemize}
Regular stationary $\das$ processes are mixing and ergodic.
\end{Th}

\begin{pf}
According to \cite{DalVJon08}, Proposition 12.4.V, $\Phi$ is mixing if
and only if for all $A,B\in\sB_0$
\[
\Q\{\phi\in\sN_0\dvtx \phi(A)>0 \mbox{ and } \phi(B+x)>0\}\to0
\qquad \mbox{as }  \|x\|\ti .
\]
%
In view of \eqref{eq:q}, choosing sufficiently large $\|x\|$ such that
$A\cap(B+x)=\varnothing$, this condition writes
\[
\int_{\M\setminus\{0\}}\bigl(1-\mathrm{e}^{-\mu(A)}\bigr)\bigl(1-\mathrm{e}^{-\mu(B+x)}\bigr) \Lambda
(\mathrm{d}\mu)\to0 .
\]
By (\ref{eq:lmreg}) and the dominated convergence theorem, $\Phi$ is
mixing if and only if $\mu(B+x)\to0$ as $\|x\|\ti$ for all
$B\in\sB_0$ and all $\mu$ from the support of the spectral measure
$\sigma$. This is clearly the case if $\sigma$ is supported by finite
measures, that is,\ for regular $\das$ processes. Similarly, the ergodicity
condition follows from \cite{DalVJon08}, Proposition 12.4.V.
\end{pf}

%
\begin{example}[(Regular stationary $\das$ processes on $\R^d$)]
\label{ex:stdas}
Consider a $\sas$ random measure $\zeta$ from
Example \ref{ex:statrm}. Using the notation from this example,
the corresponding $\das$ process $\Phi$ has the \pgfl
\[
G_\Phi[u]=\exp \biggl\{-\int_{\M_1^0}\int_{\R^d} \biggl(\int_\R
\bigl(1-u(x+y)\bigr)\mu(\mathrm{d}x) \biggr)^\alpha
\,\mathrm{d}y \,\sigma^0(\mathrm{d}\mu) \biggr\} ,
\]
where $\M_1^0$ is the set of `centred' probability measures and
$\sigma^0$ is the measure on it arising from the Haar factorisation
of the spectral measure $\sigma=\ell\times\sigma^0$.

The avoidance probabilities for $\Phi$ are given by
\[
\P\{\Phi(B)=0\}=\exp \biggl\{-\int_{\M_1^0}\int_{\R^d}
 \bigl(\mu(B-y) \bigr)^\alpha \,\mathrm{d}y\, \sigma^0(\mathrm{d}\mu) \biggr\} .
\]
The simplest case is when $\sigma^0$ is concentrated on a single
measure $\mu$. We then obtain a~$\das$ process corresponding to the
$\sas$ process from Examples \ref{ex:strm} and \ref{ex:gauss} with
$\nu$ being the Lebesgue measure $\ell$.
\end{example}

\section{Discrete stability for semigroups}
\label{sec:settings}

Let $(S,\oplus)$ be an Abelian semigroup with binary operation
$\oplus$ and the neutral element~$\neutral$. We require that $S$
possesses a (Hamel) basis $X\subset S\setminus\{\neutral\}$ such that
each element $s\in S\setminus\{\neutral\}$ is \textit{uniquely}
represented by a finite linear combination with positive integer
coefficients
%
\begin{equation}
\label{eq:basis}
s=n_1x_1\oplus\cdots\oplus n_k x_k ,
\end{equation}
where $\{x_1,\ldots,x_k\}\subseteq X$. Here and below, $ns$ stands for
the sum $s\oplus s \oplus\cdots\oplus s$ of $n\geq0$ identical elements
$s\in S$ with convention $0s=\neutral$.
Equip $X$ with the topology that makes it a locally compact second
countable space; for instance the discrete topology if $X$ is at most
countable. Note that $\{x_1,\dots,x_k\}$ with multiplicities
$\{n_1,\dots,n_k\}$ can be regarded as a finite counting measure on~$X$, so that (\ref{eq:basis}) establishes a bijection $\mathcal{I}$
between $S$ and the family of finite counting measures on~$X$. Define
a $\sigma$-algebra on $S$ as the inverse image of the $\sigma$-algebra
on the space of counting measures under the map $\mathcal{I}$. An $S$-valued random element $\xi$ is defined with respect to the
constructed $\sigma$-algebra. Then the corresponding counting measure
$\Phi(\xi)$ becomes a point process on~$X$.

The sum of random elements in $S$ is defined directly by the
$\oplus$-addition. The multiplication $t\circ\xi$ of a random element
$\xi$ by $t\in(0,1]$ is defined as the random element that has the
basis decomposition $t\circ\Phi(\xi)$, the latter obtained by
independent thinning of $\Phi(\xi)$. By Theorem \ref{th:thinprop} and
the uniqueness of the representation, the corresponding operation
satisfies $t_1\circ(t_2\circ\xi)\deq(t_1t_2)\circ\xi$ for
$t_1,t_2\in(0,1]$ and
%
\begin{equation}
\label{eq:xi-distr}
t\circ(\xi\oplus\xi')\deq(t\circ\xi)\oplus(t\circ\xi')
\end{equation}
for independent $\xi$ and $\xi'$. Note also that
$t\circ\neutral=\neutral$ for all $t\in(0,1]$.

The imposed distributivity property (\ref{eq:xi-distr}) is essential
to define the stability property on a semigroup (even for
deterministic $\xi$ and $\xi'$), and so it also rules out the
possibility of extending the scaling operation to the cases where the
decomposition (\ref{eq:basis}) is not unique or the coefficients are
allowed to be negative. For instance, assume that $a\oplus b=\neutral$
for some non-trivial $a,b\in S$, which is the case if $S$ is a group.
Then $t\circ(a\oplus b)=\neutral$, while $t\circ a\oplus t\circ b$ is
the sum of two non-trivial independent random elements and so it
cannot be $\neutral$ almost surely. This observation explains why the
classical notion of discrete stability cannot be extended to the class
of all integer-valued random variables.

Similar difficulties arise when defining discrete stability for the
max-scheme being an example of an idempotent semigroup. Indeed, if $S$
is the family of non-negative integers with maximum as the semigroup
operation, then $a\oplus a=a$. After scaling by $t\in(0,1)$ we obtain
that $t\circ a$ is the maximum of two independent random elements
distributed as $t\circ a$, which is impossible.

\begin{definition}
An $S$-valued random element $\xi$ is called \textit{discrete stable
with exponent} $\alpha$ (notation: $\das$) if
%
\begin{equation}
\label{eq:defstabs}
t^{1/\alpha}\circ\xi'\oplus(1-t)^{1/\alpha}\circ\xi''\deq
\xi
\end{equation}
for any $t\in[0,1]$, where $\xi'$, $\xi''$ and $\xi$ are independent
identically distributed.
\end{definition}

Thus, $\xi$ is $\das$ if and only if the corresponding point process
$\Phi(\xi)$ is $\das$; see Definition~\ref{def:das}. Thus, the
distributions of $\das$ random elements in $S$ are characterised by
Theorem \ref{th:coxpros}. Since the basis decomposition
(\ref{eq:basis}) is finite, $\Phi(\xi)$ is a finite point process,
meaning that the spectral measure $\sigma$ is finite and supported by
the set $\M_1$ of probability measures on $X$; see
Corollary \ref{cor:finitedas}.



A random element $\upsilon$ in $S$ is said to have Sibuya distribution
if $\Phi(\upsilon)$ is a Sibuya point process on
$X$. Theorem \ref{th:pois-sib} and the LePage representation from
Corollary \ref{cor:lpdas} immediately imply the following
result.

\begin{Th} A random element $\xi$ in $S$ is $\das$ if and only if
$\xi$ can be represented as the sum of a Poisson number of \iid\
Sibuya random elements. Alternatively, $\xi$ is $\das$ if and only
if it can be represented as an $a.s.$-finite sum $\sum_{i\geq1}
b\gamma_k^{-1/\alpha}\circ\eps_i$, where $\gamma_k,\ \eps_k$ and
$b$ are defined in (\ref{eq:lp}).
\end{Th}

When the basis $X$ is finite, (\ref{eq:basis}) establishes a
homomorphism between $S$ and $\Z_+^d$, that is,\ the family of
$d$-dimensional vectors with non-negative integer components and
addition as the semigroup operation. Thus, all random elements in
semigroups with a~finite basis can be treated as random vectors
in $\Z_+^d$; see Example \ref{ex:das-vectors}.

Now consider the case of a discrete semigroup with an infinite
countable basis $X$.

\begin{example}[(Natural numbers with multiplication)]
Consider the semigroup of natural numbers with the multiplication
operation. Its basis $X$ is the family $\sP$ of prime numbers. A
random natural number $\xi$ corresponds to a point process on $\sP$;
for example, $\xi=1$ if this point process is empty. Otherwise $\xi$ is
the multiple of prime numbers raised to their multiplicities as
\[
\xi_\Phi=\prod_{p\in\sP} p^{\Phi(p)} .
\]
Then $t\circ\xi$ is obtained by independently taking out the prime
divisors of $\xi$ with probability $1-t$. Consequently, a class of
\textit{discrete multiplicatively stable distributions} can be defined
as the distributions of $G$-valued random variables $\xi$ satisfying
\[
 (t^{1/\alpha}\circ\xi' )\cdot \bigl((1-t)^{1/\alpha
}\circ
\xi'' \bigr) =\xi ,
\]
where, as before, $\xi',\ \xi''$ are independent copies of $\xi$.
Then $\xi=\xi_\Phi$ is multiplicatively $\alpha$-stable if the
corresponding process $\Phi$ on $\sP$ has \pgfl given
by (\ref{eq:pg-pp-spec}).

Extend the domain of the \pgfl to the class of monotone pointwise
limits of the functions $u$ such that $1-u\in\mathrm{BM}(\sP)$,
allowing for infinite values of the p.g.fl., and consider $h_s\dvtx
\sP\mapsto(0,1)$ such that $h(p)=p^{s}$ for some $s$. Then
\[
\E\xi_\Phi^{s}=\E\prod_{p\in\sP}
p^{s\Phi(p)}=G_\Phi[h_s]=\exp \biggl\{-\int_{\M_1}
 \biggl(1-\sum_{p\in\sP} p^{s}\mu_p \biggr)^\alpha
\sigma(\mathrm{d}\mu) \biggr\} ,
\]
where $\M_1$ is the set of probability distributions on $\sP$. The
expression above is finite at least for all $s<0$ and can be used to
numerically evaluate the distribution of $\xi_\Phi$ by the inverse
Mellin transform.

Now consider the case of independently scattered $\Phi$, where
$\sigma$ is concentrated on degenerated distributions (see
Theorem \ref{th:dasprops}(ii)) and thus can be identified
with a sequence $\{\sigma_p,  p\in\sP\}$. Then $\Phi$ is a sequence
indexed by $p\in\sP$ of doubly stochastic Poisson random variables
$\nu_{\zeta_p}$ with parameters $\zeta_p$, which are independent
positive $\sas$ random variables with the Laplace transforms $\E
\mathrm{e}^{-h_p\zeta_p}=\exp\{-h_p^\alpha\sigma_p\}$. Now the distribution
of $\xi_\Phi$ can be explicitly characterised: if
$n=\prod_{p\in\sP} p^{k_p}$, then
\[
\P\{\xi_\Phi=n\}=\P\{\nu_{\zeta_p}=k_p\ \forall
p\in\sP\}=\prod_{p\in\sP} \E \biggl[
\frac{\zeta_p^{k_p}}{k_p!}\mathrm{e}^{-\zeta_p} \biggr] .
\]
In particular, for any $q\in\sP$ we have that
\[
\P\{\xi_\Phi=q\}=\prod_{p\neq q} \P\{\nu_{\zeta_p}=0\}\E[\zeta_q
\mathrm{e}^{-\zeta_q}]=\mathrm{e}^{-\sigma(\sP)+\sigma_q} \E[\zeta_q
\mathrm{e}^{-\zeta_q}] .
\]
If $\alpha=1/2$, then the density of $\zeta_p$ is given by
\[
f_{\zeta_p}(t)=\frac{\sigma_p}{2\sqrt{\uppi}t^{3/2}}
\exp\{-\sigma_p/(4t)\},\qquad   t\geq0 ,
\]
leading to
\[
\P\{\xi_\Phi=q\}=\mathrm{e}^{-\sigma(\sP)+\sigma_q}
\tfrac{1}{2}\sigma_q
\mathrm{e}^{-\sigma_q}= \tfrac{1}{2}\sigma_q
\mathrm{e}^{-\sigma(\sP)},\qquad   q\in\sP .
\]

The above construction can be extended for an \textit{arithmetical
semigroup} generated by a countable subset $\sP=\{p_1,p_2,
\dots\}$ called the \textit{generalised primes} -- for example, Beurling's
generalised prime numbers with the so-called Delone property, which
implies the uniqueness of the factorisation; see \cite{BatDia69,DDSG76}.
\end{example}

\begin{example}
Consider the family $S$ of all finite Abelian groups with the
semigroup operation being the direct product. The main theorem on
Abelian groups states that each such group can be uniquely
decomposed into the direct product of cyclic groups with orders
being prime numbers and their natural powers; see, for example, \cite
{EliEli85}, Theorem 7.2. Thus, the basis $X$ is the family of
prime numbers and all powers of prime numbers. The multiplication
of a cyclic group of order $p$ by real number $t\in(0,1]$ is defined
as a~random group where each factor from its decomposition is
eliminated with probability $1-t$. A spectral measure defined on the
set $\M_1$ of probability distributions on $X$ then determines the
distribution of a random stable finite group.
\end{example}

\section*{Acknowledgements}
\label{sec:acknowledgements}

The authors are grateful to two anonymous referees for their valuable
comments. I.~Molchanov was supported by the Swiss National Science Foundation,
Grant Nr. 200021-117606. S. Zuyev thanks Eddie McKenzie for introducing
him to discrete stability for positive integer random variables. I. Molchanov is
grateful to Zakhar Kabluchko for useful discussions.

\printhistory

\end{document}